\documentclass[12pt]{article}
\usepackage[left=3cm,top=3.0cm,right=3cm,bottom=2.6cm]{geometry}

\usepackage[ansinew]{inputenc}
\usepackage{amscd,amsfonts,amsmath,amssymb,amsthm,enumerate,epsfig,float,graphics,graphicx,pst-all,yhmath}
\newtheorem{lemma}{Lemma}

\newtheorem{theorem}{Theorem}

\newtheorem{remark}{Remark}

\newtheorem{conjecture}{Conjecture}

\newcommand{\fin}{\hfill $\Box$}

\title {A characterization of the ellipsoid through planar grazes}
\author{I. Gonzalez-Garc\'ia$^{1}$, J. Jer\'onimo-Castro$^{2}$, \\ E. Morales-Amaya$^{3}$,  D. J. Verdusco-Hern\'andez$^{4}$\\ 
\small{$^{1,}$$^{2}$Facultad de Ingenier\'ia}\\
\small{Universidad Aut\'onoma de Quer\'etaro, M\'exico}\\
\small{$^{3,}$$^{4}$Facultad de Matem\'aticas-Acapulco,}\\
\small{Universidad Aut\'onoma de Guerrero, M\'exico}\\
 }
\date{Dedicated to David George Larman}

\begin{document}

\maketitle
\begin{abstract}  
In this paper we proved the following: \emph{Let $K, L\subset \mathbb R^3$ be two $O$-symmetric convex bodies with $L\subset \emph{int} K$ strictly convex. Suppose that from every $x$ in $\emph{bd} K$ the graze $\Sigma(L,x)$ is a planar curve and $K$ is almost free with respect to $L$. Then $L$ is an ellipsoid.}
\end{abstract}

\section{Introduction}
There are characterizations of the ellipsoid which consider properties of the intersections of a convex body $K$ in the Euclidean space $\mathbb R^n$, $n\geq 3$, with cones or cylinders (a \textit{convex body} is a compact and convex set with non-empty interior). Given a point $x\in\mathbb R^n\setminus K$ we denote the cone generated by $K$ with apex $x$ by $\text{C}(K,x)$, i.e., $\text{C}(K,x):=\{x+\lambda (y-x): y\in K,\, \lambda\geq 0\},$ by $\text{S}(K,x)$ the boundary of $\text{C}(K,x)$, in other words, $\text{S}(K,x)$ is the support cone of $K$ from the point $x$ and by $\Sigma(K,x)$ the \emph{graze} of $K$ from $x$, i.e., 
$\Sigma(K,x):=\text{S}(K,x)\cap \partial K.$ A very especial case is when the apexes of the cones are points at infinity. In this case the grazes are called \emph{shadow boundaries} and they are obtained by intersections of $\partial K$ with circumscribed cylinders. The first proof that a convex body is an ellipsoid if and only if every shadow boundary lies in a hyperplane is due to H. Blaschke \cite{Blaschke}.  Burton \cite{burton} shows that if, for a number $\delta>0$, the graze corresponding to every point in $\mathbb R^n \backslash K$ whose distance from $K$ is less than $\delta$ is contained in a hyperplane, then $K$ is an ellipsoid. In the extraordinary paper \cite{petty} the interested reading can find more results about the ellipsoids.   
 
On the other hand, using the notion of grazes, A. Marchaud proved the following in \cite{Marchaud}: 

\emph{Let $K\subset \mathbb R^3$ be a convex body and $H$ be a plane which is either disjoint from $K$ or meets $K$ at a single point. Then $K$ is an ellipsoid if for every point $x\in H\setminus K$, the graze $\Sigma(K,x)$ contains a planar convex curve $\gamma$ such that $\emph{conv}\, \gamma \cap \emph{int}\, K \neq\emptyset$}.

However, if we change the plane $H$ for a surface $\Gamma$ which encloses  $K$ it is not known whether $K$ is an ellipsoid or not. We suspect the following is true.

\begin{conjecture}\label{penumbras_planas} Let $L\subset \emph{int}\, K\subset\mathbb R^n$ be convex bodies such that for every point $x\in\partial K$ it holds that $\Sigma(L,x)$ lies in a hyperplane. Then $L$ is an ellipsoid.
\end{conjecture} 
 
With the additional condition that the grazes are ellipses, it was proved in \cite{Ivan_et_al} that $L$ is an ellipsoid. In this work we give another progress in order to prove this conjecture, and it was very unexpected that under the hypotheses of Theorem \ref{moledeolla} the grazes of the body $L$ result to be centrally symmetric (Lemma \ref{jesus}). Before we give the statement of the main result in this work, we give some more definitions and notation. Let $L$ and $K$ be two $O$-symmetric convex bodies in $\mathbb R^3$, with $L\subset \text{int} K$. We say that the points $x,y\in \text{bd} K$ are  \textit{free with respect to $L$} if the line through $x$ and $y$, $\ell (x,y)$, does not meet $L$. Suppose that for every point $x\in\text{bd} K$, the graze $\Sigma(L,x)$ is a planar curve and denote the plane where it is contained by $\Delta_x$. The body $K$ is said to be \textit{almost free with respect to $L$} if for each $z\in \text{bd} K$  and $w\in \Pi_z \cap \text{bd} K$, where $\Pi_z$ is a plane through  $O$ parallel to $\Delta_z$, the points $z$ and $w$ are free with respect to $L$ (see Fig. \ref{reik}).

\begin{figure}[H]
\centering
\includegraphics [width=.63\textwidth]{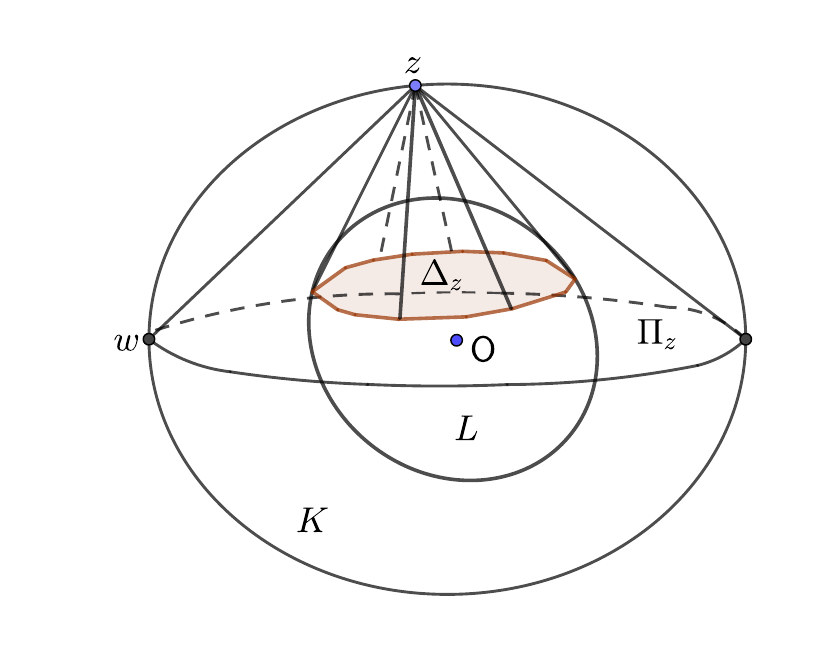}
\caption{$K$ is almost free with respect to $L$}
\label{reik}
\end{figure}

The main result of this article is the following.
 
\begin{theorem}\label{moledeolla}
Let $K, L\subset \mathbb R^3$ be two $O$-symmetric convex bodies with $L\subset \emph{int} K$ strictly convex. Suppose that from every $x$ in $\emph{bd} K$ the graze $\Sigma(L,x)$ is a planar curve and $K$ is almost free with respect to $L$. Then $L$ is an ellipsoid.  
\end{theorem}   

A very important problem in Geometric Tomography is to establish properties of a given convex body if we know some properties over its sections or its projections. An interesting conjecture was proposed by G. Bianchi, and P. M. Gruber \cite{Bianchi}: \emph{Let $K$ be a convex body in $\mathbb R^n,$ $n\geq 3$, and let $\delta$ be a continuous real function on $\mathbb S^{n-1}$ such that for each vector $u\in \mathbb S^{n-1}$ the hyperplane $\{x: \langle x, u\rangle = \delta(u)\}$ intersects the interior of $K$. If any such intersection is centrally symmetric and (with, possibly, a few exceptions) does not contain a possibly existing centre of $K$, then $K$ is an ellipsoid.}
 
It is worth to notice that Theorem \ref{moledeolla} is not only evidence for the veracity of Conjecture \ref{penumbras_planas}. Once we have the conclusion of Lemma \ref{jesus} we arrive at particular case of Bianchi and Gruber's conjecture mentioned before, and Theorem \ref{moledeolla} gives a positive answer to these cases. 
 
\section{Main result}
We first prove two lemmas.

\begin{lemma}\label{jesus}
For each $x\in \emph{bd} K$ the graze $\Sigma(L,x)$ is centrally symmetric with centre at the point $O_x:=\ell (x,-x)\cap \Delta_x$.
\end{lemma}

\begin{proof}
Let $\Omega_x:=S(L,x)\cap S(L,-x)$. By Lemma 2.2 in \cite{Jero_McAlly} we have that $\Omega_x$ is a simple and closed curve, moreover, since $L$ is centred at $O$ we have that $\Omega_x$ is centrally symmetric with centre at $O$. Let $a\in \Sigma(L,x)$ be any point and let $z$ be the point where the line $\ell (x,a)$ intersects $\Omega_x$ (see Fig. \ref{penumbras}). Consider the point $b\in \Sigma(L,x)$ such that $[a,b]$ is an affine diameter of $\Sigma(L,x)$. Suppose that $O_x\not \in [a,b]$ and let $a'$ be the point where $[-z,x]$ intersects $\Sigma(L,x)$. Let $H_a$ be a support plane of $L$ through the points $x$ and $z$, and let $\ell_a:=\Delta_x\cap H_a$, and let $\ell_z$ be the support lines of $\Sigma(L,x)$ and $\Omega_x$, through $a$ and $z$, respectively, with $\ell_z$ parallel to $\ell_a$. Let $\ell_b$ be the support line of $\Sigma(L,x)$ through $b$ and parallel to $\ell_a$ (this line exists since $[a,b]$ is an affine diameter of $\Sigma(L,x)$), and let $H_b$ be the support plane of $L$ through $\ell_b$ and $x$. 
\begin{figure}[H]
\centering
\includegraphics [width=.6\textwidth]{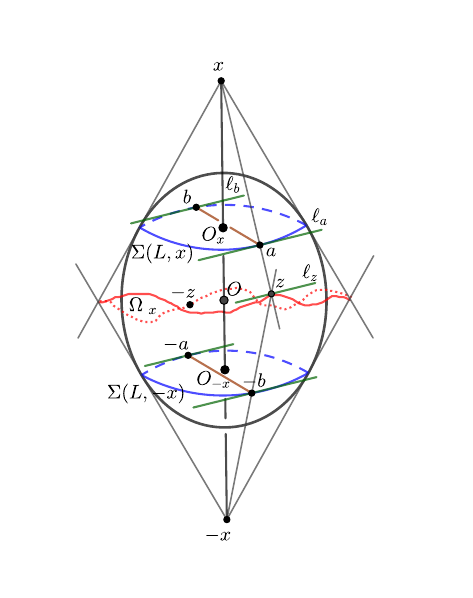}
\caption{$\Omega_x$ is a planar and closed curve}
\label{penumbras}
\end{figure}
Now, since $\Omega_x$ is an $O$-symmetric set, the line $-\ell_z$ is a support line of $\Omega_x$ through $-z$. The plane $H_{a'}$ through $-\ell_z$ and $x$ contains the line $\ell(x,-z)$, hence, the line $\ell_{a'}:=\Delta_x \cap H_{a'}$ is a support line of $\Sigma(L,x)$, parallel to $\ell_b.$ This can happen only if $a'=b$, otherwise we obtain that one of $H_{a'}$ or $H_b$ is not a support plane of $L$. We have proved that the affine diameter $[a,b]$ passes through $O_x$, and the same happens for any other affine diameter of $\Sigma(L,x)$, by a theorem of P. C. Hammer \cite{Hammer} we have that $\Sigma(L,x)$ has a centre of symmetry at $O_x$.
\end{proof}

\begin{lemma}\label{Omega_plana}
For every $x\in \emph{bd} K$ we have that $\Omega_x$ is a planar curve parallel to $\Sigma(L,x)$.
\end{lemma}

\begin{proof}
We use the notation of Lemma \ref{jesus} and Fig. \ref{penumbras}. Since the points $x,$ $b$, $-z$, $z$, $a$, are coplanar and $O_x$ and $O$ are the midpoints of the segments $[b,a]$ and $[-z,z]$, respectively, by elementary Geometry we have that $[-z,z]$ must be parallel to $[b,a]$. It follows that $\Omega_x$ is parallel to $\Sigma(L,x)$, indeed, they are homothetic with the centre of homothety at $x$.
\end{proof}

\begin{remark}
If $K$ is a Euclidean ball then we can prove at this point that $L$ is also a Euclidean ball. It is not difficult to prove that for every $x\in \emph{bd} K$, by Lemma \ref{Omega_plana}, the intersection $S(L,x)\cap S(L,-x)$ is a planar curve contained in $x^{\bot}$. It was proved in \cite{Ivan_et_al} (Theorem 2) that under this condition the body $L$ is a Euclidean ball.
\end{remark}

\begin{lemma}\label{diana}
For every $u\in \mathbb{S}^2$ there exists $v(u)\in \mathbb{S}^2$ such that for every $x\in u^{\perp}\cap \emph{bd} K$, the plane $\Delta_x$ is parallel to 
$v(u)$.
\end{lemma}

\begin{proof}
Consider a point $x\in u^{\perp}\cap \text{bd} K$, and let $y, -y\in \text{bd} K$ be such that the planes $\Delta_y,$ and $\Delta_{-y}$ are parallel to $u^{\perp}$. We claim that $\Delta_x$ is parallel to $\ell (y,-y)$. Suppose to the contrary that $\Delta_x$ is not parallel to $\ell (y,-y)$. By virtue of Lemma \ref{jesus} we have that
\begin{eqnarray}\label{mariana}
\Sigma(L,-y)=\Sigma(L,y)-2\cdot O_y.
\end{eqnarray}
Let $\mu:=\text{conv}(\Sigma(L,x)\cap \Sigma(L,y))$ be the chord with extreme points in the intersection $\Sigma(L,x)\cap \Sigma(L,y)$. Notice that the chord 
$\mu$ is well defined since $K$ is almost free with respect to $L$.
By (\ref{mariana}), the chord $\mu-2\cdot O_y$ belongs to $\Delta_{-y}$ but, since $\Delta_x$ is not parallel to $\ell (y,-y)$, this chord is not contained in $\Delta_x$.  On the other hand, by virtue that $\Delta_y$ and $\Delta_{-y}$ are at the same distance from $u^{\perp}$ and given that $\Sigma(L,x)$ has centre at $O_x\in u^{\perp}$, the image $\ell$ of $\mu$, under the central symmetry with respect to $O_x$ restricted to the plane $\Delta_x$, is in $\Delta_{-y}$ (see Fig. \ref{kari}). Applying the same argument for $-x$, it follows that $\Sigma(L,-y)$ has four parallel chords of the same length which contradicts the strict convexity of $L$. Thus $\Delta_x$ is parallel to $\ell (y,-y)$. Finally we define $v(u)$ as the unit vector parallel to $\ell (y,-y)$.
\end{proof}

\begin{figure}[H]
\centering
\includegraphics [width=.95\textwidth]{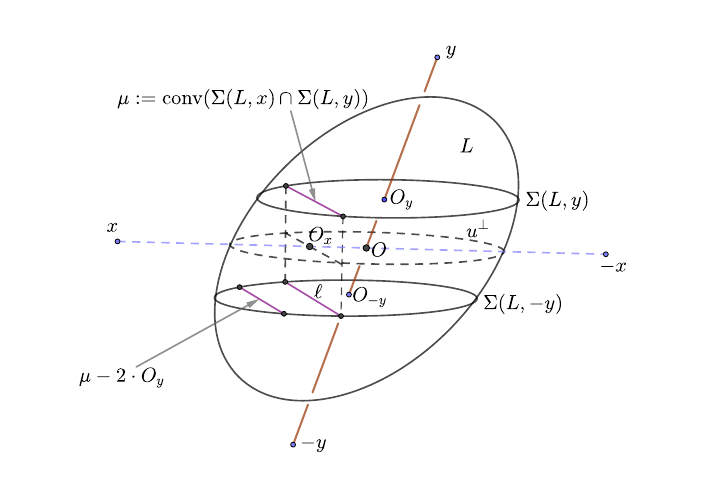}
\caption{$\Delta_x$ is parallel to $\ell (y,-y)$}
\label{kari}
\end{figure}

\emph{Proof of Theorem \ref{moledeolla}.} Let $x$ be a point in $\text{bd} K$ and we define the  unit vector $z=\frac{x}{|x|}$. Let $u \in  z^{\perp} \cap  \mathbb{S}^2$. We are going to prove that $u^{\perp} \cap \Delta_x$ is a line of affine symmetry of $\Sigma(L,x)$. In order to show this, we are going to prove that through the extreme points of the chords of $\Sigma(L,x)$ parallel to $v(u)$, $v(u)$ given by Lemma \ref{diana}, there passes support lines of $\Sigma(L,x)$ that intersect each other in a point in $u^{\perp} \cap \Delta_x$. By this property and since one of the chords pass through the centre of $\Sigma(L,x)$, we have by Lemma 3 in \cite{Jeronimo} that  $u^{\perp} \cap \Delta_x$ is a line of affine symmetry for $\Sigma(L,x)$.

Let $y\in \text{bd} (u^{\perp} \cap K)$ be a point such that line $\ell (x,y)$ is contained in $u^{\perp} \backslash C(L,x)$ (see Fig. \ref{elipse}). 
We denote by $\Gamma_1, \Gamma_2 $ the supporting planes of $L$ containing the line $\ell (x,y)$, by $a,b$ the common points between $\text{bd} L$ and $\Gamma_1, \Gamma_2 $, respectively.
Since $\Gamma_1, \Gamma_2 $ are supporting planes of $C(L,x)$ it follows $a,b\in \Sigma(L,x)$. Analogously we conclude that  $a,b\in \Sigma(L,y)$. Hence $\ell (a,b)=\Delta_x\cap \Delta_y$. By Lemma \ref{diana}, $\Delta_x$ and $\Delta_y$ are parallel to $v(u)$. Thus  the line $\Delta_x \cap \Delta_y$ is parallel to $v(u)$. We denote by $L_1,L_2$ the lines $\Gamma_1\cap \Delta_x, \Gamma_2 \cap \Delta_x$, respectively, and by $c$ the point $\ell (x,y) \cap \Delta_x$. It is clear that $c\in u^{\perp} \cap \Delta_x$ and that $L_1,L_2$ are supporting lines of $\Sigma(L,x)$ passing through $c$ and $a$ and $c$ and $b$, respectively. Varying $y\in \text{bd}(u^{\perp} \cap K)$, such that line $\ell (x,y)$ is contained in $u^{\perp} \backslash C(L,x)$ it follows that $\Sigma(L,x)$ satisfies the required property and, consequently, $u^{\perp} \cap \Delta_x$ is a line of affine symmetry of $\Sigma(L,x)$. Finally varying $u \in (  z^{\perp} \cap \mathbb{S}^2 )$ we conclude that $\Sigma(L,x)$ is an ellipse. Finally, we conclude, using Theorem 5 in \cite{Ivan_et_al}, that $L$ is an ellipsoid. \fin
 
\begin{figure}[H]
\centering
\includegraphics [width=1.0\textwidth]{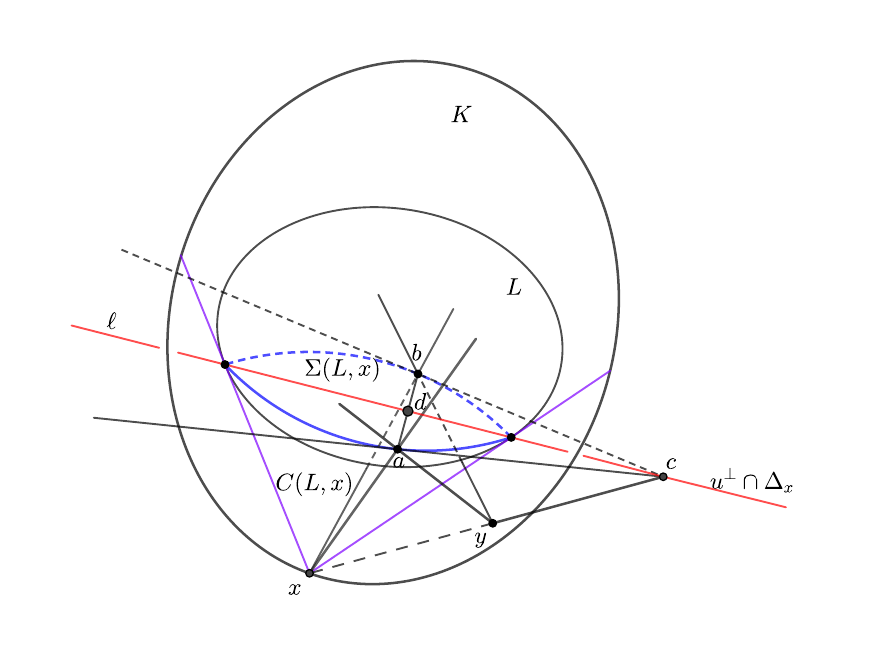}
\caption{$u^{\bot}\cap \Delta_x$ is a line of affine symmetry for $\Sigma(L,x)$}
\label{elipse}
\end{figure}

\end{document}